\documentclass[letterpaper,11pt,reqno]{amsart}

\usepackage[margin=1in]{geometry}

\let\Horig\H
\usepackage[utf8]{inputenc} 
\usepackage[T1]{fontenc}

\usepackage{hyperref}
\hypersetup{
     colorlinks   = true,
     citecolor    = black,
     linkcolor    = blue,
     urlcolor     = black
}

\usepackage{graphicx,amsmath,amssymb,amsthm,paralist,color,tikz-cd}
\usepackage{mathrsfs}
\usepackage{mathtools}
\usepackage{setspace}
\usepackage[color=blue!30]{todonotes}

\usepackage{amscd}
\usepackage{bbm}
\usepackage{hyperref}
\hypersetup{
    bookmarksdepth=2  
}

\newtheorem{thm}{Theorem}[section]
\newtheorem{lem}[thm]{Lemma}

\newtheorem{prop}[thm]{Proposition}

\newtheorem{cor}[thm]{Corollary}

\theoremstyle{definition}
\newtheorem{definition}[thm]{Definition}
\newtheorem*{definition-nono}{Definition}

\newtheorem{remark}[thm]{Remark}

\newtheorem*{acknowledgement}{Acknowledgements}

\newtheoremstyle{case}{}{}{}{}{}{:}{ }{}
\theoremstyle{case}

\newcommand{\N}{\mathbb{N}}
\newcommand{\Z}{\mathbb{Z}}

\newcommand{\R}{\mathbb{R}}

\renewcommand{\H}{\mathbb{H}}

\newcommand{\mc}{\mathcal}

\newcommand{\mbf}{\mathbf}
\newcommand{\mrm}{\mathrm}

\renewcommand{\a}{\alpha}
\renewcommand{\b}{\beta}
\newcommand{\g}{\gamma}

\renewcommand{\d}{\delta}
\newcommand{\e}{\varepsilon}
\renewcommand{\l}{\lambda}
\renewcommand{\L}{\Lambda}
\newcommand{\w}{\omega}
\newcommand{\s}{\sigma}
\newcommand{\vp}{\varphi}
\renewcommand{\t}{\tau}
\renewcommand{\th}{\theta}
\newcommand{\z}{\zeta}

\newcommand{\set}[1]{\left\{#1\right\}}
\renewcommand{\r}{\rightarrow}

\def\multiset#1#2{\ensuremath{\left(\kern-.3em\left(\genfrac{}{}{0pt}{}{#1}{#2}\right)\kern-.3em\right)}}
\newcommand{\norm}[1]{\left\lVert#1\right\rVert}

\newcommand{\Ccal}{\mc{C}}
\newcommand{\Dcal}{\mc{D}}

\newcommand{\supp}{\mrm{supp}}

\numberwithin{equation}{section}

\title{L\texorpdfstring{\textsuperscript{2}}{2}-Flattening of Self-similar Measures on Non-degenerate  Curves}

\author{Amir Algom}
\address{Department of Mathematics, The University of Haifa at Oranim, Tivon 36006, Israel
}
\email{amir.algom@math.haifa.ac.il
}
\author{Osama Khalil}
\address{Department of Mathematics, Statistics, and Computer Science, University of Illinois Chicago, IL.}
\email{okhalil@uic.edu}

\AtBeginDocument{%
   \def\MR#1{}
}

\begin{document}

\begin{abstract}
    Let $\mu$ be a non-atomic self-similar measure on $\R$, and let $\nu$ be its pushforward to a non-degenerate curve in $\R^d, d\geq 1$. We show that for every $\e>0$ there is $p>1$, so that $\norm{\hat{\nu}}_{L^p(B(R))}^p = O_\e(R^\e)$ for all $R>1$, where $B(R)$ is the $R$-ball about the origin.  As a corollary, we show that convolution with $\nu$ quantitatively improves $L^2$-dimension.
\end{abstract}

\maketitle

\section{Introduction}
\subsection{Statement of Main results}
Let $U\subseteq \mathbb{R}$ be an open interval. We call a $C^{d+1}$ curve  $Q:U\r \R^d $  \textit{non-degenerate} if
\begin{align}\label{eq:nondegenerate}
        \det [Q^{(1)}(x) Q^{(2)}(x) \cdots Q^{(d)}(x)] \neq 0 \text{ for all } x\in U.
\end{align}
 When $Q$ is real analytic, we call it \textit{non-trapped} if the determinant above has at most finitely many zeros in any compact interval; this is equivalent to the trace of $Q$  not being  contained in a proper affine hyperplane of $\mathbb{R}^d$ (Lemma \ref{lem:good intervals higher dimensions}).
 Such curves have become standard objects in harmonic analysis, see e.g. \cite{fassler2014restricted, gan2022restricted} for some discussion and applications in modern projection theory.

The purpose of this paper is to study the $L^p$-norm of the Fourier transform of  self-similar measures on $\mathbb{R}$, when they are pushed-forward by either a non-trapped or a   non-degenerate curve.

\begin{thm} \label{main theorem general}
Let $\mu \in \mathcal{P}(\mathbb{R})$ be a non-atomic  self-similar measure, and let $d\geq 2$. Let  $U$ be an open interval containing  $\supp(\mu)$, and let $Q:U\r \R^{d} $ be either a non-trapped analytic curve, or a $C^{d+1}$ non-degenerate curve. 
    Let $\nu =Q\mu$ be the pushforward of $\mu$ via $Q$.
Then,
\begin{equation}\label{eq:main thm flattening}
\forall \e>0\quad \exists p>1 \quad \forall R >0:\, \norm{ \hat{\nu} }^p_{L^p (B(R))} = O_\e(R^\e),
\end{equation}
where $B(R)$ is the $R$-ball around $0$ in $\R^d$.
\end{thm}
Recall that a self-similar measure $\mu$ on the line is a Borel probability measure satisfying the stationarity condition, for some strictly positive probability vector $\mathbf{p}$,
$$\mu = \sum_{i=1} ^n \mbf{p}_i \cdot f_i \mu,\, \text{ where all } f_i\in \text{Aff}(\mathbb{R)} \text{ are s.t. } |f_i'|\in (0,1), \text{ and }  f_i\mu \text{ is the push forward of } \mu \text{ by } f_i.$$
See Section \ref{Section self-similar prel.} for  more discussion about them. Also, the notation $O_\e$ means the implicit constant may depend on $\e$; it may also depend on $\mu,d,Q$ and the maps $f_i$. Next, we note that the non-degeneracy condition \eqref{eq:nondegenerate}, and the non-trapped condition, cannot be relaxed from the Theorem. Indeed, \eqref{eq:main thm flattening} cannot hold for any measure $\nu$ living on a proper affine subspace, since then $|\hat{\nu}(\xi)|\sim 1 $ for all frequencies $\xi$ that are nearly orthogonal to the subspace supporting $\nu$. In particular, such frequencies contribute non-trivially to the $L^p$-norm of $\hat{\nu}$.

Finally, let us explain the relation between the title of the paper and our main result.
For $m\in \N$, let $\Dcal_m$ be the dyadic partition of $\R^d$ given by translates of $2^{-m}[0,1)^d$ by $2^{-m}\Z^d$. Given $q\geq 1$, we define the moment sum corresponding to $m,q,$ and a Borel probability measure $\nu$, by
\begin{align*}
    s_m(\nu,q) \stackrel{\mrm{def}}{=} \sum_{Q\in \Dcal_m} \nu(Q)^q. 
\end{align*}
For $q=\infty$, we set
\begin{align*}
    s_m(\nu,\infty) \stackrel{\mrm{def}}{=} \max \set{\nu(Q): Q\in \Dcal_m}.
\end{align*}
A standard computation (Lemma \ref{lem:Fourier vs moment sum}) shows that if $\nu$ is compactly supported, then for every $R>1$, letting $m = [\log_2 R]\in \N$ be the integer part of $\log_2 R$, we have
    \begin{align*}
         s_m(\nu,2)  
        \asymp 2^{-dm} \norm{\hat{\nu}}^2_{L^2(B(R))}.
    \end{align*}
    Thus, the conclusion of Theorem \ref{main theorem general} is equivalent to the following flattening phenomenon of moment sums:
for every $\e>0$ there exists $p\in\N$ such that for all $m$ large enough,
\begin{equation} \label{eq: flattening}
s_m(\nu^{\ast p},2) =  O_\e\left( 2^{m(\e-d)} \right).    
\end{equation}
Here $\mu\ast\nu$ means the usual Euclidean convolution of measures (pushforward of $\mu\times \nu$ by $(x,y)\mapsto x+y$), and $\nu^{\ast p}$ means the $p$-fold self-convolution of $\nu$.

Theorem~\ref{main theorem general} has the following consequences. Recall that   for a Borel probability  $\nu$ on $\mathbb{R}^d$ and $1<q<\infty$, one defines its $L^q$-dimension via
$$\dim_q (\nu)=\frac{\tau_q(\nu)}{q-1}, \quad \text{ where } \tau_q (\nu) = \liminf_{m\rightarrow \infty} \frac{-\log s_m(\nu,q)}{m}.$$
For $q=\infty,$ we set $\dim_\infty(\nu) = \t_\infty(\nu)$. The $L^\infty$-dimension is also known as the \textit{Frostman exponent} of $\nu$.
With this notation, we have the following corollary of Theorem~\ref{main theorem general}.
\begin{cor} \label{main corollary}
    Let $\nu$ be as in Theorem~\ref{main theorem general}. Then,
    \begin{enumerate}
        \item For all $q\in [2,\infty]$ we have $\lim_{n\rightarrow \infty} \dim_q(\nu^{\ast n})= d$.

        \item For every $\g>0$, there is $\eta=\eta(\g)>0$ and $m_0 = m_0(\g,\eta)>1$ such that the following holds for all integers $m\geq m_0$:
        For every Borel probability measure $\th$, 
        \begin{align}\label{eq:L2 improving}
      s_m(\th,2) > 2^{m(\g-d)} \quad \Longrightarrow  \quad s_m(\theta * \nu) 
        \leq 2^{-\eta m} s_m(\th,2).
        \end{align}
    \end{enumerate}
    Indeed, both statements hold true for any measure $\nu$ satisfying~\eqref{eq:main thm flattening}.
\end{cor}
\begin{proof}
Part (1) follows directly for $q=2$ from Theorem \ref{main theorem general}, in its equivalent form \eqref{eq: flattening}. 
Hence, the case $q=\infty$ then follows by Young's inequality; cf.~\cite[Lemma 5.2]{MosqueraShmerkin}, which then yields the claim for all other values of $q$.
Part (2) follows from Theorem \ref{main theorem general} similarly to the proof of~\cite[Theorem 4.1]{MosqueraShmerkin}.    
\end{proof}

We proceed now to discuss some prior results in harmonic analysis, fractal geometry, and dynamical systems, and how Theorem \ref{main theorem general} and its proof compare to them. 

\subsection{Prior results} We begin by comparing Theorem \ref{main theorem general} to recent results in harmonic analysis. Recall that a probability measure $\nu$ on $\mathbb{R}^d$ is called an $s$-Frostman measure, where $s\in [0,d]$, if $\mu(B(x,r))\leq r^s$ for all $x\in \mathbb{R}^d$ and $r>0$. In \cite{Orponen2023para}, Orponen proved that if $\nu$ is an $s$-Frostman measure on the truncated parabola $\mathbb{P}:=\lbrace (x,x^2):\,[-1,1]\rbrace$, then $\norm{\hat{\nu}}_{L^4 (B(R))} ^4 \ll R^{2-2s}$ for all $R\geq 1$. For $s\in (0,1)$ and $p>4$  he was able to obtain an $\e=\e(p,s)$ improvement $\norm{\hat{\nu}}_{L^p (B(R))} ^p \ll R^{2-2s-\e}$. Orponen also related this problem to the sum-product phenomenon and the Borel subring problem \cite[Section 1.2]{Orponen2023para}. He conjectured that for every $s\in [0,1]$ and $\e>0$ there exists some $p=p(\e,s)\geq 1$ such that for every $s$-Frostman measure on $\mathbb{P}$, $\norm{\hat{\nu}}_{L^p (B(R))} ^p \ll R^{2-\min \lbrace 3s, 1+s \rbrace+p\e }$, see \cite[Conjecture 1.6]{Orponen2023para}. Furthermore, exploiting the relation with iterated sum-sets, it was shown in \cite[Example 1.8]{Orponen2023para} that the threshold $\min \lbrace 3s, 1+s \rbrace$ cannot be further improved in this generality.  Dasu and Demeter \cite{Dasu2024demeter} later extended these results for $s$-Frostman measures $\nu$  that are supported on the graph of a $C^3( [-1,1],\, \mathbb{R)}$ function $\gamma$ such that $\min_{x\in [-1,1]} |\gamma''(x)|>0$; they established an Orponen-like bound of the form $\norm{\hat{\nu}}_{L^6 (B(R))} ^p \ll R^{2-2s-\beta}$ where $\beta=\beta(s), s\in (0,1)$. The conjectured bound from \cite[Conjecture 1.6]{Orponen2023para} was then established by Orponen, Puliatti, and Py{\"o}r{\"a}l{\"a} \cite{orponen2024fourier2}, along with further applications and analogies with the dimensions of iterated sum-sets on $\mathbb{P}$; some problems in this direction, however,   still remain open \cite[Questions 1 and 2]{orponen2024fourier2}. Finally, for measures $\nu$ supported on graphs of functions as in the work of Dasu and Demeter \cite{Dasu2024demeter}, it was very recently shown by Demeter and Wang \cite{Demeter2025Wang} that if $\nu$ is $s$-Frostman where $s\in (0,\frac{1}{2})$ then $\norm{\hat{\nu}}_{L^6 (B(R))} ^6 \ll R^{2-2s-\frac{s}{4}+\e}$ for all $\e>0$. See also \cite{yi2024bounded} for further research in this direction.

To compare Theorem \ref{main theorem general} to these results, we first note that all non-atomic self-similar measures are $s$-Frostman; this was first proved by Feng and Lau \cite{Feng2009Lau} (Proposition \ref{prop:mu is Frostman}). The precise computation of the best possible $s$ given the generating IFS is a subtle problem, that is still open in general; see Shmerkin's work \cite{shmerkin2016furstenberg} for some recent progress on it. Nonetheless, Theorem \ref{main theorem general} dramatically improves upon the results in \cite{Orponen2023para, Dasu2024demeter, orponen2024fourier2, Demeter2025Wang} as it shows optimal flattening regardless of the precise value of $s$,  for arbitrary non-trapped or non-degenerate curves in every dimension. In particular, we show that for self-similar measures the general threshold \cite[Example 1.8]{Orponen2023para} can be  substantially improved. However, unlike  \cite{Orponen2023para, orponen2024fourier2} where the key step involves a reduction to a problem about Furstenberg sets, or \cite{Dasu2024demeter, Demeter2025Wang} that use various decoupling arguments, the stationary structure of self-similar measures allows for a wider array of tools. It is thus interesting to ask how much regularity the measure  should enjoy so that it can break the   general threshold given in \cite[Example 1.8]{Orponen2023para}. In the same spirit, one may ask whether Theorem \ref{main theorem general} holds true for non-atomic Ahlfors-David regular measures (note, though, that not all self-similar measures have this property).

Let us now place Theorem \ref{main theorem general} within recent literature on fractal geometry and dynamical systems. It can be considered a variant of the Fourier decay problem, which asks about optimal \textit{pointwise} estimates for the decay rate of $\hat{\nu}(\xi)$ when $\nu$ is a stationary measure. For self-similar measures, it is of central importance to understand when have polynomial decay ($\hat{\nu}(\xi)=O(\norm{\xi}^{-\alpha})$ for some $\alpha>0$) due to e.g. the relation between this property and the absolute continuity of $\nu$ \cite{Shmkerin2014Abs}. Solomyak \cite{solomyak2019fourier}, extending the classical Erd\Horig{o}s-Kahane argument,  had shown that polynomial decay is generic among self-similar measures, and a full characterization of the Rajchman property for them had been given by Li-Sahlsten \cite{li2018fourier}, Br\'{e}mont \cite{bremont2019rajchman} (see also Varj{\'u}-Yu \cite{varju2020fourier}), and Rapaport \cite{rapaport2021rajchman}. However, the only known \textit{explicit} examples of self-similar measures with polynomial  decay are due to Dai, Feng, and Wang \cite{Dai2007Feng}, a work that was recently extended by Streck \cite{streck2023absolute}. We remark that there are many explicit examples of measures with logarithmic decay ($\hat{\nu}(\xi)=O((\log |\xi|)^{-\alpha})$ for some $\alpha>0$) under various Diophantine conditions, see e.g. \cite{li2018fourier, varju2020fourier,  algom2021decay, khalil2024polynomial}.

Another closely related direction concerns pointwise decay rates for non-linear push-forwards of self-similar measures. It was first observed by Kaufman \cite{Kaufman1984ber} that, among other things, if $g$ is any $C^2$ diffeomorphism on $[0,1]$ such that $g''>0$ then the pushforward of the Cantor-Lebesgue measure on the middle thirds Cantor set does have polynomial Fourier decay (even though the original measure is not even Rajchman). This was extended to all uniformly contracting self-similar measures by Mosquera and Shmerkin \cite{MosqueraShmerkin}, and then to all non-atomic self-similar measures by Algom, Chang, Meng Wu, and Yu-Liang Wu \cite{Algom2023Wu}, and simultaneously and independently by Baker and Banaji \cite{baker2024polynomial}. We note that, combined with the earlier works of Algom, Rodriguez Hertz, and Wang \cite{algom2023polynomial} and Baker and Sahlsten \cite{Baker2023Sahl}, this shows that all self-conformal  measures with respect to non-affine real analytic IFSs have polynomial Fourier decay. These results were extended to higher dimensions in various directions by Algom, Rodriguez Hertz, and Wang \cite{algom2024plane}, Baker, Khalil, and Sahlsten \cite{khalil2024polynomial}, and Banaji and Yu \cite{BanajiYu}.

The proof of Theorem \ref{main theorem general} is related to the approach of \cite{Algom2023Wu,khalil2024polynomial}, where
the main tool exploited is the following large deviations estimate for the Fourier transform. Let $\mu$ be a non-atomic self-similar measure on $\mathbb{R}$. Then,
\begin{equation} \label{eq: Tsujii into}
\forall \e>0\, \exists \delta>0\text{ s.t. } \forall T\gg 1, \text{ we can cover } \lbrace|\xi|\leq T:\, |\hat{\mu}(\xi)|\geq T^{-\delta}\rbrace \text{ by } O_\e(T^\e) \text{ intervals of size } 1.    
\end{equation}
This was first observed by Kaufman \cite{Kaufman1984ber} for some Bernoulli convolutions. Tsujii \cite{Tsujii2015self} proved \eqref{eq: Tsujii into} for all non-atomic self-similar measures. Mosquera and Shmerkin \cite{MosqueraShmerkin} proved effective (quantitative) versions of \eqref{eq: Tsujii into}; these in turn allow one to give explicit lower bounds on the Fourier dimension of non-linear images of self-similar measures; see also \cite{BanajiYu} for related results.
Note that large deviation estimates of the form \eqref{eq: Tsujii into} are equivalent to the $L^p$-norm bounds in \eqref{eq:main thm flattening}.

A general criterion implying that \eqref{eq: Tsujii into} holds for arbitrary Borel probability measures (not necessarily self-similar) was established in \cite{khalil2023exponential}.
Namely, it is shown that $\nu$ satisfies \eqref{eq:main thm flattening} if for all $\e>0$, there is $\d>0$ such that for all proper affine subspaces $W$, we have
\begin{align}\label{eq:anc}
    \nu(W(\d r)\cap B(x,r)) \leq \e \, \nu\left(B(x,O(r))\right),
\end{align}
for all $x\in \supp(\nu)$ outside an exponentially small exceptional set, and all but a small proportion of scales $r>0$; cf. \cite[Cor. 1.7 and 6.4]{khalil2023exponential} for precise statements.
Here, $W(\d r)$ and $B(x,r)$ denote the $\d r$-neighborhood of $W$ and the $r$-ball around $x$ respectively.

The proof this result involves producing conditions on measures that are \textit{not} $L^2$-improving in the sense of \eqref{eq:L2 improving}, generalizing Shmerkin's $1$-dimensional inverse theorem \cite{shmerkin2016furstenberg} to higher dimensions; cf.~\cite[Prop. 11.10]{khalil2023exponential}. 
In particular, it is shown that large subsets in the supports of such measures must locally concentrate near proper subspaces at many scales.
A similar result was recently obtained by Shmerkin in \cite{shmerkin2025inverse}, where the local structure of the non-improved measure $\th$ in \eqref{eq:L2 improving} was also described. In both instances, the proof relies on Hochman's inverse theorem for entropy \cite{Hochman-InverseEntropy} and the asymmetric Balog-Szemer\'edi-Gowers Lemma \cite{TaoVu-book}.
To obtain \eqref{eq:main thm flattening}, this multi-scale concentration is ruled out using \eqref{eq:anc} through induction on scales \cite[Cor. 6.4]{khalil2023exponential}.

A natural question is to find weaker non-concentration conditions than \eqref{eq:anc} under which \eqref{eq: Tsujii into} (and hence \eqref{eq:main thm flattening}) can be shown to hold.
This question served as one of the major motivations of this work.
In this vein, Theorem \ref{main theorem general} provides a natural class of examples where the local non-concentration estimate \eqref{eq:anc} fails on large sets at every scale (due to local concentration of curves along their tangent lines), while the flattening estimate \eqref{eq:main thm flattening} holds.
In particular, such local concentration is a serious obstacle to carrying out the approach of \cite{khalil2023exponential}, and indeed, our proof of Theorem \ref{main theorem general} uses different techniques.

\subsection{On the proof of Theorem~\ref{main theorem general}} 
\label{sec:proof outline}
The proof of Theorem \ref{main theorem general} consists of two main parts, corresponding to two ranges of frequencies which we now introduce.
For $R\geq 1$ and $\e>0$,
    let 
    \begin{align*}
        C_{R,\e} &\stackrel{\mrm{def}}{=} \set{(\th,\zeta)\in \R\times \R^{d-1}: |\th|^\e \leq \norm{\zeta}\leq R},
        \nonumber\\
        E_{R,\e} & \stackrel{\mrm{def}}{=}\set{(\th,\zeta)\in \R\times \R^{d-1}: \norm{\zeta}\leq |\th|^\e \leq  R^\e}.
    \end{align*}
In particular, we have the following decomposition of $B(R)$:
\begin{align}\label{eq:decompose B(R)}
    B(R)  = C_{R,\e} \bigcup E_{R,\e}.
\end{align}
To describe the first ingredient, let $d\geq 1$ and define the moment curve 
$$V_d(x)\stackrel{\mrm{def}}{=} (x,x^2,\dots , x^d) \text{ if } d\geq 2, \qquad \text{ and }\qquad  V_1(x)=x \text{ otherwise}. $$
\begin{thm}\label{thm: case 1}
    Let $\mu$ be a non-atomic self-similar measure  on $\R$. 
    Let $d\geq 2$ and assume that
    \begin{align}\label{eq:flattening of moment d-1}
        \forall \e>0\quad \exists p>1 \quad \forall R >0:\, \norm{ \widehat{V_{d-1}\mu} }^p_{L^p (B(R))} = O_\e(R^\e).
    \end{align}
    Let $g:U\r \R^{d-1} $ be a map defined on an open neighborhood $U$ of $\supp(\mu)$, such that $Q(x)=(x,g(x))$ is either a non-trapped analytic curve, or a $C^{d+1}$ non-degenerate curve.
    Let $\nu =Q\mu$ be the pushforward of $\mu$ to the graph of $g$.
    Then, for every $\e>0$, there is $p_0=p_0(\mu,g,\e)>1$ such that for all $p\geq p_0$, we have
    \begin{align*}
        \int_{C_{R,\e}} |\widehat{\nu}(\xi)|^p\;d\xi =O_{\e}(R^\e).
    \end{align*}
\end{thm}
The following is the second main ingredient in our proof which handles the region $E_{R,\e}$.
\begin{prop} \label{prop: Case 2}
Let $\mu \in \mathcal{P}(\mathbb{R})$ be a non-atomic  self-similar measure and $d\geq 2$. 
Let $g:U\r \R^{d-1} $ be a  $C^1$-map defined on an open neighborhood $U$ of $\supp(\mu)$.
    For $Q(x)=(x,g(x))$, 
    let $\nu =Q\mu$ be the pushforward of $\mu$ to the graph of $g$.
Then, for every $\e >0$, there exists  $p=p(\e, \mu)>1$ such that for all $R\geq 1$, we have
\begin{align*}
    \int_{E_{R,\e}}\left| \widehat{\nu} (\xi) \right|^{p} \, d\xi = O_{\e,\mu}(R^\e).
\end{align*}
\end{prop}

First, let us show how Theorem \ref{main theorem general} follows quickly from Theorem \ref{thm: case 1} and Proposition \ref{prop: Case 2}.

\begin{proof}
    [\textbf{Proof of Theorem \ref{main theorem general} assuming Theorem \ref{thm: case 1} and Prop.~\ref{prop: Case 2}}]

    We proceed by induction on the ambient dimension $d$.
    By Tsujii's Theorem (Corollary \ref{Coro Tsujii}) for $d=2$, and by induction for $d>2$, we may assume that Hypothesis~\eqref{eq:flattening of moment d-1} holds for $\nu=V_{d-1}\mu$, where $V_{d-1}$ is the moment curve in $\R^{d-1}$.
    By replacing $\mu$ with an affine image of smaller diameter, and replacing $U$ with a smaller neighborhood if necessary, we may assume, after an affine change of coordinates, that $Q(x)$ takes the form $(x,g(x))$, for a map $g:U\r\R^{d-1}$ as in Theorem~\ref{thm: case 1}.
    Hence, in view of~\eqref{eq:decompose B(R)}, Theorem \ref{thm: case 1} and Proposition \ref{prop: Case 2} together imply Theorem \ref{main theorem general}.
\qedhere
\end{proof}

Next, we describe the ideas behind the above two intermediate results, beginning with Proposition \ref{prop: Case 2}.
Here, the main observation is that, via the coordinate relation in $E_{R,\e}$ and the Lipchitz continuity of $\hat{\nu}$, we can fully reduce it to an estimate about the Fourier transform of $\mu$ itself (Lemma \ref{lem:annular partition of Case 2}). This estimate is then obtained in Lemma \ref{Lemma 3 rewrite}, where the end-game step follows from Tsujii's large deviations estimate \eqref{eq: Tsujii into} (Corollary \ref{Coro Tsujii}). The details are given in Section \ref{Section:decay near first coordinate}.

The key to the proof of Theorem \ref{thm: case 1} is the following proposition providing \textit{polynomial pointwise decay} in a large region of of frequencies, from which Theorem \ref{thm: case 1} follows immediately; cf.~Corollary \ref{cor:integrate Case 1 bound} for this deduction.

 \begin{prop}\label{prop:Case 1 general curves}
     Let $\mu$ be a non-atomic self-similar measure  on $\R$. 
    Let $d\geq 2$ and assume that the flattening estimate~\eqref{eq:flattening of moment d-1} holds for $V_{d-1}\mu$.
    Let $g:U\r \R^{d-1} $ be as in Theorem~\ref{thm: case 1}. 
    For $Q(x)=(x,g(x))$, 
    let $\nu =Q\mu$ be the pushforward of $\mu$ to the graph of $g$. Then, there exists $\g=\g(\mu,g)>0$ such that for every $\xi =(\th,\zeta)\in \R\times \R^{d-1}$ with $\zeta\neq \mathbf{0}$, we have $|\hat\nu(\xi)|\ll \norm{\zeta}^{-\g}$.        
    \end{prop}

Proposition \ref{prop:Case 1 general curves} is closely related to the aforementioned pointwise Fourier decay results for non-linear pushforwards of self-similar measures obtained in \cite{Algom2023Wu,baker2024polynomial}, as well as to forthcoming work of Banaji and Yu for pushforwards of self-similar measures on $\R^k$ to graphs of analytic maps $g:\R^k\r \R^{k+d}$. 
We give a somewhat different proof of this result here, which we believe will be useful for future generalizations of Theorem \ref{main theorem general}.

It remains to explain the proof of Proposition \ref{prop:Case 1 general curves}.  
First, via Taylor expansion and a change of coordinates, we will deduce Proposition~\ref{prop: Case 2} from the special case $Q(x)=V_d(x)$, where $V_d(x)$ is the moment curve.
This deduction is carried out in Section \ref{sec:case 1 general curves}.

The special case of the moment curve is carried out in Proposition \ref{prop:Case 1 Veronese all dims multiple contractions}. 
The key observation in the proof of the latter,  which is based upon a construction of Feng and K{\"a}enm{\"a}ki   \cite[Lemma 3.1]{feng2018self}, is that the pushed self-similar measure $\nu =V_d\mu$ is, in fact, self-affine: that is, there exist $A_i \in \mrm{GL} (\mathbb{R}^d)$ with $\norm{A_i}<1$, and $b_i\in \R^d$, such that for $F_i(\mbf{x})=A_i\mbf{x}+b_i$, we have
$$\nu = \sum_{i=1} ^n \mbf{p}_i \cdot F_i \mu.$$
See Lemma \ref{Lemma self affine} below for the precise statement.
Note that the probability vector $\mathbf{p}$ is the same one that is used to define $\mu$.

This observation is then used to express the Fourier coefficient $\hat{\nu}(\xi)$ as an \textit{average} over many Fourier coefficients of the original measure $\mu$.
Roughly speaking, since the derivative of $V_d$ is essentially given by a copy of $V_{d-1}$, the set of frequencies we average over arise from a \textit{projection} of $V_{d-1}\mu$ in a suitable direction determined by the original frequency $\xi$. 

We then take advantage of our inductive flattening hypothesis \eqref{eq:flattening of moment d-1} to ensure that \textit{all} projections admit a uniform lower bound on their \textit{upper Frostman exponent}.
Indeed, the latter property in fact holds for any measure satisfying \eqref{eq:flattening of moment d-1}; cf.~Theorem \ref{prop:flattening implies friendly}. 
    It is likely that such Frostman bounds could be verified directly without appealing to~\eqref{eq:flattening of moment d-1}; cf.~\cite[Theorems 2.1(b) and 2.3]{KleinbockLindenstraussWeiss} for special cases of this statement.

The Frostman bound on the projections of $V_{d-1}\mu$ in turn ensures that
our average is sampled along a well-separated set of frequencies, thus enabling us to apply Tsujii's large deviations estimate \eqref{eq: Tsujii into} to conclude the proof.

\subsection{Notation} 
Throughout the article, given two quantities $A$ and $B$, we use the Vinogradov notation $A\ll B$ to mean that there exists a constant $C\geq 1$, possibly depending on the self-similar measure $\mu$, ambient dimension $d$, and the analytic map $g$,  such that $|A|\leq C B$. In particular, we suppress these dependencies except when we wish to emphasize them.
We write $A\ll_{x,y} B$ to indicate that the implicit constant depends on parameters $x$ and $y$.
We also write $A=O_x(B)$ to mean $A\ll_x B$.
Finally, we write $A \asymp_x B$ to mean $A\ll_x B$ and $B\ll_x A$.
For $x\in \R$, 
    \begin{align*}
        e(x) \stackrel{\mrm{def}}{=} e^{2\pi i x}.
    \end{align*}
 Finally, for $v\in\R^n$, $\norm{v}$ denotes the $\max$ norm.

\begin{acknowledgement}
The authors thank Tuomas Orponen for many discussions regarding this project. We also thank Pablo Shmerkin for discussions regarding the inverse Theorem \cite{shmerkin2025inverse}, and Amlan Banaji for comments on an earlier draft.
 O.K. acknowledges NSF support under grants  DMS-2337911 and DMS-2247713. A.A. is supported by Grant No. 2022034 from the United States - Israel Binational Science Foundation (BSF), Jerusalem, Israel.

\end{acknowledgement}

\section{Preliminaries}

\subsection{Self-similar measures} \label{Section self-similar prel.}
 Let 
$$\Phi = \set{f_i(x) =\l_i x + t_i}_{i\in I}, \, |I|=n>1,$$
be a finite set of invertible strictly contracting similitudes preserving a compact interval $J\subseteq \mathbb{R}$ (often we will work with $J=[0,1]$). In particular, $0<|\lambda_i|<1$ for all $i\in I$. We call $\Phi$ a \textit{self-similar IFS} (Iterated Function System). It is well known that there exists a unique compact non-empty set $K=K_\Phi \subseteq J$ such that 
$$K = \bigcup_{i\in I} f_i(K).$$
The set $K$ is called the \textit{attractor} of $\Phi$, and the \textit{self-similar set} generated by it. We always make the assumption that there exist $i,j \in I$ such that the fixed point of $f_i$ differs from that of $f_j$; this is known to imply that $K$ is infinite (in fact, has positive Hausdorff dimension). 

Next, let $\mathbf{p}=(\mbf{p}_1,\dots,\mbf{p}_n)$ be a (strictly positive) probability vector: For all $i$ we have $\mbf{p}_i>0$, and $\sum_i \mbf{p}_i =1$. Then  it is well known that there exists a unique Borel probability measure $\mu$ such that
$$\mu = \sum_i \mbf{p}_i \cdot f_i \mu,\, \text{ where } f_i\mu \text{ is the push forward of } \mu \text{ by } f_i.$$
The measure $\mu$ is called a \textit{self-similar measure}, and is supported on $K$. Occasionally we will use the notation \textit{weighted IFS} and write the pair $(\Phi,\mbf{p})$ to indicate a self-similar IFS paired with a probability vector. Under our assumptions ($K$ is infinite and $\mathbf{p}$ is strictly positive) the measure $\mu$ is known to be non-atomic. Thus, all self-similar measures considered in this paper are non-atomic. 

We next define cut-sets related with $\Phi$ and $\mu$. For more details about this, see e.g. \cite[Chapter 2]{bishop2013fractal}. For $\omega=\omega_1\dots\omega_n\in I^n,n\in \mathbb{N},$ we define $f_\omega := f_{\omega_1} \circ \dots f_{\omega_n}$.
\begin{definition}
    \label{def:cut sets}
    Let $\Phi = \set{f_i(x) =\l_i x + t_i}_{i\in I}$ be a self-similar IFS on $\R$, and let $\t\in (0,1)$. The cut-set corresponding to $\t$ is defined by
    $$P_\t := \left\lbrace \omega=(\omega_1,\dots,\omega_k) \in I^*: \, \left|\l_{\omega_1}\cdot \dots \l_{\omega_k}\right|<\t \text{ yet } \left|\l_{\omega_1}\cdot \dots \l_{\omega_{k-1}}\right|\geq \t \right\rbrace.$$
    Given a probability vector $\mbf{p}$, and a self-similar measure $\mu$ associated to $(\Phi,\mbf{p})$, we set
      \begin{align*}
       \mu_\t \stackrel{\mrm{def}}{=} \sum_{\w\in P_\t} \mathbf{p}_\w \d_{t_\w},
    \end{align*}
    where
    $$t_\w= f_\omega (0) \text{ for } \omega \in P_\t.$$
\end{definition}
We record a number of standard facts regarding cut-sets. See e.g. \cite[Lemma 2.2]{Algom2023Wu} for a closely related discussion.
\begin{lem}\label{lem:cut set}
    Let $\mu$ be a non-atomic self-similar measure on $\mathbb{R}$ with respect to the weighted IFS $(\Phi,\mbf{p})$. If $\t\in (0,1)$ is sufficiently small then:
\begin{enumerate}
    \item We have 
    $$\mu = \sum_{\w\in P_\t} \mathbf{p}_\w \cdot f_\w\mu, \text{ where } \mathbf{p}_\omega = \prod_{i=1} ^{|\omega|} \mathbf{p_i},\, \text{ and } f_\omega = f_{\omega_1}\circ \dots \circ f_{\omega_{|\omega|}}.$$
    In particular,
    $$\widehat{\mu} = \sum_{\w\in P_\t} \mathbf{p}_\w \cdot \widehat{f_\w\mu}.$$
    
    \item 
    Let $\L_\t \subseteq (-1,1)$ denote the set 
    $$\L_\t:=\lbrace r_\omega:\, \omega \in P_\t\rbrace, \, \text{ where } r_\omega = r_{\omega_1}\cdot \dots r_{\omega_{|\omega|}}.$$
    Then, $\# \L_\t \ll (-\log \t)^{n+1}$, where $n=|\Phi|$.
\end{enumerate}    
\end{lem}
\begin{proof}
The proof of Part (1) follows from a standard stopping time type argument, see e.g. \cite[Definition 2.2.3 and Lemma 2.2.4]{bishop2013fractal}.

As for the numerical estimate in Part (2), let us first estimate, for $m\in \mathbb{N}$
$$\left| \left\lbrace r_{\eta}:\, \eta\in \mathcal{I}^m \right\rbrace \right|.$$
That is, we count how many different contraction ratios the maps in the IFS 
$$\Phi^m:= \lbrace f_\eta:\, \eta\in \mathcal{I}^m \rbrace$$
can admit. Note that for any map $f_\eta(x)=r_\eta \cdot x +t_\eta \in \Phi^m$, $r_\eta$ only depends on the amount of times each $r_i$ appears in $\eta$, for $i\in \mathcal{A}$. So, writing 
$$n_i = n_i (\eta) = \left| \left\lbrace 1\leq j \leq m:\, \eta_j = i\right\rbrace \right|,$$
we have 
$$\log \left| r_\eta \right| = \sum_{i=1} ^n n_i \log \left| r_{i} \right|,\quad \text{ and } \sum_{i=1} ^n n_i = m.$$
By a standard combinatorial argument, there are at most
$$\binom{n+m-1}{m}$$
different possible values for this sum. It follows that
$$\left| \left\lbrace r_{\eta}:\, \eta\in \mathcal{A}^m \right\rbrace \right| \leq O\left( m^n \right),$$
where the implicit constant is bounded independently of $m$.

Finally, since $\min_{i\in \mathcal{A}} |r_i|>0$,  there exists some $d\in  \mathbb{N}$ such that: For all $\t>0$ and all $\omega \in \mathcal{A}^\mathbb{N}$,
$$\max \left\lbrace \,  |\omega|:\, \omega \in P_\t \right\rbrace \leq -d\cdot \log \tau.$$
Thus, by definition 
$$ \mathcal{P}_\t  \subseteq \bigcup_{i=1} ^{[-\log \t]+1} \left\lbrace r_{\eta}:\, \eta\in \mathcal{I}^{d\cdot i} \right\rbrace,$$
where $[-\log \tau]$ is the integer part of $-\log \tau$. So, via our previous estimates,
$$\left| \left\lbrace r_{\eta}:\, \eta\in \mathcal{P}_\t \right\rbrace \right| \leq \sum_{i=1} ^{[-\log\t]+1} O\left( (d\cdot i)^n \right) \leq ([-\log \t]+1)\cdot O\left( (d\cdot ([- \log\t]+1))^{n} \right)  = O\left( (- \log\tau) ^{n+1} \right).$$
This is the required bound.    
\end{proof}

\subsection{Frostman exponents of projections and discretizations}
Let $\mu$ be a non-atomic self-similar measure on $\R$ for an IFS $\Phi$ and a strictly positive probability vector $\mathbf{p}$.
We require the following uniform estimate on the Frostman exponent of \textit{projections} of pushforwards of the discrete measures $\mu_\t$ to non-degenerate curves.
     For an  affine subspace $W\subset \R^d$ and $\e>0$, we denote by $W(\e)$ be the open $\e$-neighborhood of $W$.
\begin{prop}\label{Prop:discrete friendly measures}
Let $\mu\in \mathcal{P}(\mathbb{R})$ be a non-atomic self-similar measure. Let $d\geq 1$ and let $g:U\r \R^d$ be a Lipschitz continuous map defined on an open neighborhood $U$ of $\supp(\mu)$. Let $Q(x)=(x,g(x))$ and $\t\in (0,1)$ and consider the measures $\nu,\,\nu_\t \in \mathcal{P}(\mathbb{R}^{d+1})$ defined by
$$\nu:=Q\mu, \quad \text{ and } \quad \nu_\t:=Q\mu_\t.$$
    Suppose that $\nu$ satisfies \eqref{eq:main thm flattening}.
    Then, there are $\b,\varrho>0$ and $C\geq 1$ such that for every proper affine subspace $W\subset \R^d$  we have 
    \begin{align*}
        \nu_\t(W(\e)) \leq C\e^\b,
        \quad \text{ for all } \e >\t^\varrho.
    \end{align*}
\end{prop}

Proposition~\ref{Prop:discrete friendly measures} is a rather direct consequence of the following Theorem:

\begin{thm}[{\hspace{.1pt}\cite[Theorem 6.23]{khalil2023exponential}}]
\label{prop:flattening implies friendly}
    Let $\nu$ be a compactly supported probability measure on $\R^d$ satisfying~\eqref{eq:main thm flattening}. 
    Then, there are $\b>0$ and $C\geq 1$ such that $\nu(W(\e)) \leq C\e^\b$ for all proper affine subspaces $W\subset \R^d$ and $\e>0$.
\end{thm}
\begin{remark}
    The statement of Theorem~\ref{prop:flattening implies friendly} is different from the reference~\cite[Theorem 6.23]{khalil2023exponential}, however the proof of the latter is written for measures satisfying~\eqref{eq:main thm flattening}.
\end{remark}

\begin{proof}[\textbf{Proof of Proposition \ref{Prop:discrete friendly measures}}]
    First, since the IFS is uniformly contracting, for some $\g>0$ we have
    $$|\mu_\t(\vp) -\mu(\vp)| \ll \t^\g\norm{\vp}_{\mrm{Lip}},$$
    for any Lipschitz function $\vp$, where $\norm{\vp}_{\mrm{Lip}}$ is the Lipschitz constant of $\vp$.
    This follows for instance by straightforward adaptation of the argument of~\cite[Theorem 4.4.1(ii)]{hutchinson1981fractals} to averages over general cut-sets.
    Hence, the same bound holds for $\nu$ and $\nu_\t$ in place of $\mu$ and $\mu_\t$. To conclude the proof, let $B\subset \R^d$ be a large ball containing the supports of $\nu$ and $\nu_\t$ for all $\t\in (0,1)$.
    Let $\vp$ be a Lipschitz function that is identically $1$ on $W(\e)\cap B$ and vanishing outside $W(2\e)$. In particular, $\norm{\vp}_{\mrm{Lip}}\ll \e^{-1}$. 
    Combined with Theorem ~\ref{prop:flattening implies friendly}, we obtain
    \begin{align*}
        \nu_\t(W(\e)) \leq \nu(\vp) + O(\e^{-1}\t^\g)
        \leq \nu(W(2\e)) + O(\e^{-1}\t^\g)
        =
        O(\e^\b + \e^{-1}\t^\g),
    \end{align*}
    for some $\a>0$.
    The above bound is thus $O(\e^\b)$ whenever $\e> \t^{\g/(1+\b)}$.
\end{proof}

We also recall the following fact, due to Feng and Lau \cite{Feng2009Lau}, that non-atomic self-similar measures are always upper Frostman (H{\"o}lder) regular. See \cite{gorodetski2022h} for a more recent and more general version. 
\begin{prop}[{\hspace{.1pt}\cite[Proposition 2.2]{Feng2009Lau}}]
\label{prop:mu is Frostman}
Let $\mu$ be a non-atomic self-similar measure on $\mathbb{R}$. Then, there exists some $s_0=s_0(\mu)>0$ such that
$$\sup_{x\in \mathbb{R}} \mu\left( B(x,r) \right) \ll_{\mu} r^{s_0}, \quad  \text{ for all } r>0. $$
In particular, there is $\varrho >0$ such that for every $\t \in (0,1)$, we have that 
\begin{align*}
    \sup_{x\in \mathbb{R}} \mu_\t\left( B(x,r) \right) \ll_{\mu} r^{s_0}, \quad \text{ for all } r>\t^\varrho,
\end{align*}
where $\mu_\t$ is the discretization of $\mu$ defined in Def.~\ref{def:cut sets}.
\end{prop}

\begin{remark}
The reference~\cite[Proposition 2.2]{Feng2009Lau} proves the first assertion of Prop.~\ref{prop:mu is Frostman}.
The second assertion concerning the discrete measures $\mu_\t$ follows from the first by the same argument in the proof of Proposition~\ref{Prop:discrete friendly measures}, or from \cite{gorodetski2022h}.
\end{remark}

\subsection{Tsujii's large deviations estimate and its consequences}
In this section, we recall a result of Tsujii that plays a key role in our analysis.
\begin{thm} [{\hspace{.1pt}\cite{Tsujii2015self}}] \label{Theorem Tsujii}
Let $\mu$ be a non-atomic self-similar measure on $\mathbb{R}$. Then for every $\e>0$ there exists some $c_0=c_0(\e,\mu)$ such that for all $t\gg 1$
$$\text{Leb} \left( \left\lbrace \xi\in (-e^t,\,e^t):\,|\hat{\mu}(\xi)|\geq |\xi|^{-c_0} \right\rbrace \right) \leq e^{t\e}.$$
\end{thm}
Theorem \ref{Theorem Tsujii} immediately implies the following version of Theorem \ref{main theorem general} for the self-similar measure $\mu$ itself.

\begin{cor} \label{Coro Tsujii}
Let $\mu$ be a non-atomic self-similar measure on $\mathbb{R}$ and let $B(R)$ be the $R$-ball around $0$.  Then,
$$\forall \e>0\quad \exists p\in \mathbb{N}:\, \left \lVert \hat{\mu} \right\rVert_{L^p (B(R))} ^p = O_\e(R^\e).$$
Moreover, for every $\e>0$, there is $\d>0$ so that the set of $\xi\in B(R)$ with $|\hat{\mu}(\xi)|> R^{-\d}$ can be covered with $O_{\e,\mu}(R^\e)$ intervals of length $1$.
\end{cor}
\begin{proof}
Let $\e>0$ and let $p\in \mathbb{N}$ to be chosen later. Applying Theorem \ref{Theorem Tsujii} with $t=\log R$ and letting $c_0$ be as in the Theorem, we see that
$$\int_{-R} ^R |\hat{\mu}(\xi)|^p\, d\xi \leq \int_{ \lbrace \xi\in B(R):\, |\hat{\mu}(\xi)|< |\xi|^{-c_0} \rbrace} |\hat{\mu}(\xi)|^p\, d\xi+R^\e \leq \int_{-R} ^R |\xi|^{-c_0\cdot p}\, d\xi+R^\e = O(R^{-c_0\cdot p+1})+R^\e.$$
Taking $p\gg 1$ in a manner that depends on $c_0$, we obtain the first conclusion.
The second conclusion follows from Theorem~\ref{Theorem Tsujii} using the fact that $\hat{\mu}$ is Lipschitz and hence is slowly varying on small intervals; cf.~\cite[Corollary 2.5]{Algom2023Wu} for a detailed proof.
\end{proof}

\subsection{Self-similar sets and measures on the moment curve}
A key observation in our analysis is that if we are push a  self-similar set (resp. measure) to the moment curve $(x,x^2,...,x^\ell)$ then the resulting set is, in fact, a self-affine set (resp. measure).  The proof is based upon a construction of Feng and K{\"a}enm{\"a}ki   \cite[Lemma 3.1]{feng2018self}:

\begin{lem} \label{Lemma self affine}
Suppose $K_\Phi$ is a self-similar set generated by the IFS $\Phi = \lbrace f_i(x)=\lambda_i\cdot x+t_i\rbrace_{i=1} ^n$ on $[0,1]$. Let $P_\ell(x)=(x,x^2,...,x^\ell)$. Then $P_\ell(K)$ is the attractor of the self affine IFS
$$\Psi:= \left\lbrace F_i(x) = A_i\cdot x -A_i\left(-\frac{t_i}{\lambda_i},\,\left( -\frac{t_i}{\lambda_i} \right)^2,\dots,\,\left( -\frac{t_i}{\lambda_i} \right)^\ell \right) \right \rbrace,$$
where
$$A_i = \begin{pmatrix}
            \l_i c_{i,1,1} & 0 & 0 & \dots & 0 \\ \l_i ^2 c_{i,2,1} &\l_i ^2 c_{i,2,2} & 0 & \dots & 0\\ \l_i ^3 c_{i,3,1} & \l_i ^3 c_{i,3,2} & \l_i ^3 c_{i,3,3} & \dots &0\\ \vdots & \vdots & \vdots & \ddots & \vdots \\ \l_i ^\ell c_{i,\ell,1} & \l_i ^\ell c_{i,\ell,2}& \l_i ^n c_{i,\ell,3}& \dots & \l_i ^\ell c_{i,\ell,\ell}
        \end{pmatrix},$$
        $$ \text{ and } c_{i,k,j}:=\binom{k}{j}\cdot \left(\frac{t_i}{\l_i}\right)^{k-j} \text{ for }  (i,k,j)\in \lbrace 1,...,n \rbrace \times \lbrace1,...,\ell \rbrace^2.$$
\end{lem}
We remark that a direct computation shows  that if $\max_{i} |\lambda_i|<\frac{1}{2^{2\ell} \sqrt{\ell}}$ then $||T_i||<1$ for all $i$, whence the IFS $\Psi$ will be uniformly contracting. Clearly we can induce the original IFS to achieve this (assuming as we may that $\ell$ is given). So, in application, we can assume without the loss of generality that $\Psi$ is a uniformly contracting IFS.
\begin{proof}
Let $x\in \mathbb{R}$. Then for every $1\leq k \leq \ell$ and $1\leq i \leq n$ we have
$$\left(x-\left(-\frac{t_i}{\l_i}\right)\right)^k = \sum_{j=1} ^k c_{i,k,j} \left( x^j-\left(-\frac{t_i}{\l_i}\right)^j\right).$$
It follows that for every $1 \leq i \leq n$ and $x\in \mathbb{R}$ we have
\begin{eqnarray*}
F_i(x,x^2,\dots,x^\ell) &=& T_i \left( x-\left(-\frac{t_i}{\l_i}\right),\, x^2-\left(-\frac{t_i}{\l_i}\right)^2,\dots,\, x^\ell-\left(-\frac{t_i}{\l_i}\right)^\ell \right)     \\
&=&  \left( \l \left( x-\left(-\frac{t_i}{\l_i}\right)\right),\, \l^2 \left( x^2-\left(-\frac{t_i}{\l_i}\right) \right)^2,\dots,\, \l^\ell \left( x-\left(-\frac{t_i}{\l_i}\right) \right)^\ell \right) \\
&=&  \left( \l_i\cdot x+t_i,\,  \left( \l_i\cdot x+t_i \right)^2,\dots,\,  \left( \l_i\cdot x+t_i \right)^\ell \right).
\end{eqnarray*}
It follows that, writing $\tilde{K}=P_\ell(K)$, for every $i\in \lbrace 1,...,n\rbrace$, we have
$$T_i(\tilde{K}) = \left\lbrace \left(f_i(x),\ f_i ^2 (x),\dots, f_i ^\ell (x)\right): x\in K \right\rbrace$$
Therefore,
$$\bigcup_{i=1} ^n F_i(\tilde{K}) = \bigcup_{i=1} ^n \left\lbrace \left( f_i(x),\,f_i^2(x),\,\dots,f_i ^\ell(x)\right):\,x\in K \right\rbrace = \lbrace P_\ell(x): x\in K \rbrace = P_\ell(K)=\tilde{K}.$$
The proof is complete.
\end{proof}

\begin{cor} \label{cor self affine measure}
Let $\mu \in \mathcal{P}([0,1])$ be a self-similar measure with respect to the  self-similar IFS $\Phi = \lbrace f_i(x)=\lambda_i\cdot x+t_i\rbrace_{i=1} ^n$ and the probability vector $\mathbf{p}$. Then $\nu=P_\ell\mu$ is a self-affine measure with respect to the IFS $\Psi$ defined in Lemma \ref{Lemma self affine}, and the same probability vector $\mathbf{p}$.
\end{cor}
\begin{proof}
This is a direct consequence of the previous Lemma and its proof.
\end{proof}

\subsection{Basic properties of analytic maps}

The following lemma is a consequence of finiteness of the number of zeros of non-identically vanishing analytic maps in a compact interval, combined with upper Frostman regularity of self-similar measures.

\begin{lem}\label{lem:good intervals higher dimensions}
        Let $g:U\r \R^{d-1} $ be an analytic map defined on an open neighborhood $U$ of $[0,1]$, and so that its graph is not contained in a proper affine subspace of $\R^d$.
        Let $G(x)$ denote the $((d-1)\times (d-1))$-matrix $[g^{(2)}(x) g^{(3)}(x) \cdots g^{(d)}(x)]$.
        Then, for all sufficiently small $\d>0$, depending on $g$, there exists a set $E = E(\d)\subseteq [0,1]$ such that:
\begin{enumerate}
    \item $E$ is a union of $O_g(1)$-many intervals.

    \item There exist $c_1\geq 1$, depending only on $g$, such that 
    $$\min \set{|\det(G(x) )|: x\in [0,1]\setminus E} \gg_{g} \d^{c_1} ,$$ 
    where $\det$ denotes the matrix determinant.
\end{enumerate}

\end{lem}
\begin{proof}
    Since $x\mapsto D(x)\stackrel{\mrm{def}}{=}\det(G(x))$ is analytic, it either vanishes identically, or has finitely many zeros in $[0,1]$.
    Moreover, it is known (cf.~\cite{WronskianLinearDependenceAnalytic}) that identical vanishing of $D(x)$ is equivalent to the graph of $g$ being contained in a proper affine subspace, and thus, $D$ can only vanish at finitely many points.
    Let $Z=\set{ x\in [0,1]: D(x)=0}$ be this finite set.
Let $Z^{(\d)}$ be the $\d$-neighborhood of $Z$. Then, $Z^{(\d)}$ is a union of $O_g(1)$ intervals for all small enough $\d$.
Here, the implicit constant depends only on (the number and order of vanishing of zeros of) $D$. Finally, it is a standard consequence of the properties of zeros of real analytic functions that there exists some $c_1= c_1(g)\geq 1$ such that 
$$\min \set{ |D(y)|: y\in [0,1]\setminus Z^{(\d)}   } \gg_{g} \d^{c_1} .$$
This completes the proof by taking $E = Z^{(\d)}$.
\qedhere
\end{proof}

\subsection{Moment sums vs Fourier transforms}

Recall the definition of the dyadic partitions $\Dcal_m$ and moment sums given in the introduction.
For $x\in \R^d$, we denote by $\Dcal_m(x)$ the unique element of $\Dcal_m$ containing $x$. 

The following lemma is a consequence of Plancherel's theorem, and relates Fourier analytic properties of measures to moment sums of their discretizations.
\begin{lem}
\label{lem:Fourier vs moment sum}
    Let $\mu$ be a compactly supported Borel probability measure on $\R^d$.
    Then, for every $R>1$, letting $m = [\log_2 R]\in \N$ be the integer part of $\log_2 R$, we have
    \begin{align*}
         s_m(\mu,2)  
        \asymp 2^{-dm} \norm{\hat{\mu}}^2_{L^2(B(R))},
    \end{align*}
    where $B(R)$ denotes the $R$-ball around the origin.
\end{lem}
\begin{proof}
    The proof follows similar lines to~\cite[Proof of Claim 2.8]{FengNguyenWang}, where the inequality $$2^{-dm}\norm{\hat{\mu}}^2_{L^2(B(R))}
        \ll s_m(\mu,2)  $$ was essentially proved for $d=1$; cf.~\cite[Eq. (6.28) and (6.29)]{khalil2023exponential}.
    It can also be deduced by a very similar argument to the one we give below for the reverse inequality.

    Let $\vp$ be a Schwartz function on $\R^d$ satisfying $\vp \geq 1$ on the unit ball $B(1)$, $\hat{\vp} \geq 0$, and $\supp(\hat{\vp}) \subset B(1)$; cf.~\cite[Example 3.2]{Mattila2015new} for a construction of such function.
    For $R>1$, let $\vp_R(x) = R^d \vp(Rx)$.
    Then, $\widehat{\vp_R}(\xi) = \hat{\vp}(\xi/R)$.
    
    Note that $\widehat{\vp_R} \ll_\vp 1$ on $\supp(\widehat{\vp_R})\subset B(R)$.
    Hence, by Plancherel's formula, we have
    \begin{align*}
        \int_{\R^d} |\vp_R \ast \mu|^2 \,dx \simeq
        \int_{\R^d} |\widehat{\vp_R \ast \mu}|^2 \,d\xi 
        = \int_{\R^d} |\widehat{\vp_R}|^2  |\hat{\mu}|^2 \,d\xi 
        \ll \int_{B(R)} |\hat{\mu}|^2\;d\xi.
    \end{align*}
    On the other hand, we have
    \begin{align*}
        \vp_R \ast \mu(x) = R^d \int \vp(R(x-y)) \,d\mu(y)
        \geq R^d \mu(B(x,1/R)).
    \end{align*}
    And, hence, we get
    \begin{align}\label{eq:avg squared measure ll L2 Fourier}
        R^{2d} \int \mu(B(x,1/R))^2 \,dx \ll \norm{\hat{\mu}}^2_{L^2(B(R)}.
    \end{align}

    It remains to bound the left-hand side of~\eqref{eq:avg squared measure ll L2 Fourier} from below using a suitable moment sum.
    To this end, let $m=[\log_2 R] \in \N$, and let $m' = m+O_d(1)\in \N$ be such that for every cube $P\in \Dcal_{m'}$ and every $x\in P$, we have
    \begin{align*}
        B(x,1/R) \supseteq P.
    \end{align*}
    This yields the lower bound
    \begin{align*}
        \int \mu(B(x,1/R))^2 \,dx
        = \sum_{P\in \Dcal_{m'}}  \int_P \mu(B(x,1/R))^2 \,dx
        \gg 2^{-dm'} \sum_{P\in \Dcal_{m'}} \mu(P)^2 ,
    \end{align*}
    where we used the fact that each $P$ has Lebesgue measure $\asymp 2^{-dm'}$ in the last inequality.
    Moreover, by the Cauchy-Schwartz inequality, and the fact that each $Q\in \Dcal_m$ contains $O_d(1)$ boxes $P\in \Dcal_{m'}$, we get
    \begin{align*}
        \sum_{P\in \Dcal_{m'}} \mu(P)^2 
        = \sum_{Q\in \Dcal_m} \sum_{P\in \Dcal_{m'}, P\subseteq Q} \mu(P)^2
        \gg_d  \sum_{Q\in \Dcal_m} \left(\sum_{P\in \Dcal_{m'}, P\subseteq Q} \mu(P)\right)^2
        = \sum_{Q\in \Dcal_m} \mu(Q)^2.
    \end{align*}
    Combining the above estimates, we obtain
    \begin{align*}
        R^{2d} 2^{-dm'} s_m(\mu,2) \ll \norm{\hat{\mu}}^2_{L^2(B(R)}.
    \end{align*}
    This implies the desired inequality $s_m(\mu,2) \ll 2^{-dm}\norm{\hat{\mu}}^2_{L^2(B(R))}$ since $m'=m+O_d(1)$.
    \qedhere

\end{proof}

\section{Uniform pointwise Fourier decay away from the first coordinate} \label{Section away from first coordinate}
The goal of this Section is to prove Theorem~\ref{thm: case 1}.
The key step, Proposition~\ref{prop:Case 1 general curves}, is to show that for a self-similar measure pushed to a non-degenerate curve $\nu$ as in Theorem \ref{main theorem general}, there is an $\a>0$ such that:
$|\hat\nu(\th,\zeta)|\ll \norm{\zeta}^{-\a}$ for every $\xi =(\th,\zeta)\in \R\times \R^{d-1}$, with $\zeta\neq \mathbf{0}$. In the general context of Theorem \ref{main theorem general}, this will take care of the case when $\norm{\zeta}\gg|\th|^\e$. We begin with the  moment curve.

\subsection{The case of moment curves}

Let $d\geq 2$ and recall the moment curves $V_d$: 
$$V_d(x)= (x,x^2,\dots , x^d).$$

\begin{prop} \label{prop:Case 1 Veronese all dims multiple contractions}
    Let $\mu$ be a non-atomic self-similar measure for a uniformly contracting weighted IFS $(\Phi,\mbf{p})$ on $\R$. 
    Let $d\geq 2$ and assume that $V_{d-1}\mu$ satisfies the flattening estimate~\eqref{eq:flattening of moment d-1}.
    Let $\nu=V_d\mu$.
    Then, there exists $\a>0$ such that for every $\xi =(\th,\zeta)\in \R\times \R^{d-1}$ with $\zeta\neq \mathbf{0}$, we have $|\hat\nu(\xi)|\ll \norm{\zeta}^{-\a}$.
\end{prop}

\begin{proof}[\textbf{Proof of Proposition~\ref{prop:Case 1 Veronese all dims multiple contractions}}]
    By replacing $\mu$ with an affine image of itself, we shall assume its support is contained in $[0,1]$.
    Let $\Phi=\set{f_i:f_i(x)=\l_ix+t_i }_{i\in I}$.
    Apply Corollary \ref{cor self affine measure} to find the IFS $\Psi=\set{F_i:i\in I}$ on $\R^d$ so that each $F_i$ is of the form $F_i(x)=A_ix+v_i$ for some $v_i\in \R^d$, and for some contracting lower triangular matrices $A_i$ as in the corollary.
    
    Let $\xi=(\th,\zeta)\in \R\times \R^{d-1}$ be a frequency with $\norm{\zeta}>1$.
    Fix a parameter $\varrho\in (0,1/2)$ to be chosen using Lemma~\ref{lem:apply discrete friendly} below.
    By abuse of notation, we use $P_\zeta$ to denote 
    the cut-set $P_\t$ defined in Definition ~\ref{def:cut sets} for $\t = \norm{\zeta}^{-(1-\varrho)}$. That is,
    $$P_\zeta \stackrel{\mrm{def}}{=} P_{\norm{\zeta}^{-(1-\varrho)}}$$
    Let $\L_\zeta=\set{\l_\w:\w\in P_\zeta}$ denote the set of all contraction ratios of the maps in the one-dimensional IFS $\Phi$ corresponding to words in $P_\zeta$. For $\l\in \L_\zeta$, define a measure $\g_\l$ on $\R^{d-1}$ by 
    \begin{align*}
        \g_\l = \sum_{\w\in P_\zeta, \l_\w=\l} \mbf{p}_{\w} \d_{(2t_\w, 3t_\w^2,\dots, dt_\w^{d-1})} .    
    \end{align*}
     Note that $\g_\l$ is not a probability measure in general, as in may have total mass less than $1$.
    Denote by $\check{\g}_\l$ the image of $\g_\l$ under the map $x\mapsto -x$.
    \begin{lem} \label{Lemma first bound}
     We have
    \begin{align*}
        |\hat{\nu}(\xi)|^{2}
        \leq      
        \#\L_\zeta^2 \times \sum_{\l\in \L_\zeta}
        \int |\hat{\mu}(\l \zeta \cdot y)| \;d (\g_\l\ast \check{\g}_\l)(y)
        +  O(\# \L_\zeta^2 \times \norm{\zeta}^{2\varrho-1}),
    \end{align*}
     where $\g_\l\ast \check{\g}_\l$ is the additive convolution of the two measures.
    \end{lem}
\begin{proof}
By Lemma~\ref{lem:cut set}, we obtain
    \begin{align*}
        \hat{\nu}(\xi)
        = \sum_{\w \in P_\zeta} \mbf{p}_\w \int  e( \langle \xi , F_\w(x)\rangle) \;d\nu(x)
        = \sum_{\l\in \L_\zeta}
        \int 
        \left(\sum_{\substack{\w \in P_\zeta, \l_\w=\l}} 
        \mbf{p}_\w  e( \langle \xi , F_\w(x)\rangle) \right) \;d\nu(x).
    \end{align*}
    Hence, by Cauchy-Schwarz and Jensen inequalities,
    \begin{align}
        |\hat{\nu}(\xi)|^2 
        \leq \#\L_\zeta^2 \times \sum_{\l\in \L_\zeta}
        \int \left|\sum_{\substack{\w \in P_\zeta, \l_\w=\l}} 
        \mbf{p}_\w  e( \langle \xi , F_\w(x)\rangle) \right|^2 \;d\nu(x).
    \end{align}
    Expanding the square, we obtain
    \begin{align*}
        |\hat{\nu}(\xi)|^2 
        \leq \#\L_\zeta^2 \times \sum_{\l\in \L_\zeta}
        \sum_{\substack{\w_1,\w_2\in P_\zeta, \\ \l_{\w_1}=\l=\l_{\w_2}}}
        \prod_{j=1}^2 \mathbf{p}_{\w_j}
       \left| \int   e( \langle \xi (A_{\w_1}-A_{\w_2}), x\rangle )  \;d\nu(x)\right|.
    \end{align*}

    Let $\Ccal^1_{\w_1,\w_2}$ denote the first column of the matrix $A_{\w_1}-A_{\w_2}$.
    By Lemma~\ref{Lemma self affine}, we have for all $x\in \supp(\nu)$, and all $\l\in \L_\zeta$, and $\w_1,\w_2\in P_\zeta$ with $\l_{\w_1}=\l=\l_{\w_2}$, that
        \begin{align*}
            \langle \xi (A_{\w_1}-A_{\w_2}), x\rangle
            =
             (\zeta \cdot \Ccal^1_{\w_1,\w_2} ) x_1 + O(\l^2 \norm{\zeta}) ).
        \end{align*}
    It follows that, as the projection of $\nu$ on the first coordinate is our original measure $\mu$,
      \begin{align*}
         \left| \int   e( \langle \xi (A_{\w_1}-A_{\w_2}), x\rangle )  \;d\nu(x)\right|
         = |\hat{\mu}(\zeta \cdot \Ccal^1_{\w_1,\w_2} ) + O(\l^2 \norm{\zeta}) ).
      \end{align*}

    To simplify notation, for $\vec{\w} = (\w_1,\w_{2})\in P_\zeta^2$, we write $\mbf{p}_{\vec{\w}} = \prod_{j=1}^{2} \mbf{p}_{\w_j}$.
    Recall that for all $\l\in \L_\zeta$, we have $\l \asymp \norm{\zeta}^{-(1-\varrho)}$.
    Hence, by combining the last estimate with the fact that the Fourier transform is Lipschitz continuous, we obtain
    \begin{align*}
        |\hat{\nu}(\xi)|^{2}
        &\leq      
        \#\L_\zeta^2 \times \sum_{\l\in \L_\zeta}
        \sum_{\substack{(\w_1,\w_2)\in P^2_\zeta, \\ \l_{\w_1}=\l=\l_{\w_2}}}
        \mathbf{p}_{\vec{\w}}
        |\hat{\mu}\left( \zeta \cdot \Ccal^1_{\w_1,\w_2}\right) |
        +  O(\# \L_\zeta^2 \times \norm{\zeta}^{2\varrho-1}).
    \end{align*}

    Next,  for $\zeta=(\zeta_1,\dots, \zeta_{d-1})$, we have by Lemma \ref{Lemma self affine}
    \begin{align*}
        \zeta \cdot \Ccal^1_{\w_1,\w_2}
        =  \sum_{k=2}^d \l k(t_{\w_1}^{k-1} - t_{\w_2}^{k-1}) \zeta_{k-1} .
    \end{align*}    
    For $\l\in \L_\zeta$ recall the definition of the measures $\g_\l$ and $\check{\g}_\l$ given before the Lemma. 
Then, the above bound can be rewritten as follows
    \begin{align*}
        |\hat{\nu}(\xi)|^{2}
        \leq      
        \#\L_\zeta^2 \times \sum_{\l\in \L_\zeta}
        \int |\hat{\mu}(\l \zeta \cdot y)| \;d (\g_\l\ast \check{\g}_\l)(y)
        +  O(\# \L_\zeta^2 \times \norm{\zeta}^{2\varrho-1}),
    \end{align*}
This was our claim. \qedhere   
\end{proof}
We will need a uniform Frostman estimate on the projections of the various measures $\g_\l$. This is how we select the value of $\varrho>0$ that appears in the definition of $P_\zeta$, and it is the only point in the proof where the assumption on $V_{d-1}\mu$ is used.
\begin{lem}\label{lem:apply discrete friendly}
    There are constants $C\geq 1,$  $\b>0$ and $0<\varrho <1/2$, depending only on $\mu$ and $d$, such that $\g_\l(W(\e))\leq C\e^\b$ for all proper affine subspaces $W$ and all $\e>\norm{\zeta}^{-\varrho}$.
    \end{lem}
    \begin{proof}
        By our assumption, 
        $$\forall \e>0\quad \exists p>1 \quad \forall R >0:\, \norm{ \widehat{V_{d-1}\mu} }^p_{L^p (B(R))} = O_\e(R^\e).$$
        Since the curve $$\widetilde{V}_{d-1}(x)=(2x,3x^2,\dots, dx^{d-1})$$ is a linear image of $V_{d-1}$, it follows that the measure $\tilde{\nu}\stackrel{\mrm{def}}{=}\widetilde{V}_{d-1}\mu$ also satisfies this property. 
        Hence, by Proposition~\ref{Prop:discrete friendly measures}, there are $\b>0$ and $\varrho >0$ such that for every $\t>0$ the discrete measure 
        \begin{align*}
            \tilde{\nu}_\t =  \sum_{\w\in P_\t} \mbf{p}_{\w} \d_{(2t_\w, 3t_\w^2,\dots, dt_\w^{d-1})} 
        \end{align*}
        give mass $O(\e^\b)$ to $\e$-neighborhoods of proper affine subspaces whenever $\e > \t^\varrho$.
        Moreover, without loss of generality, the parameter $\varrho$ may be taken $<1/2$.
        Note that $\g_\l(A) \leq \tilde{\nu}_\t(A)$ for all Borel sets $A$ and all $\l\in \L_\zeta$. Putting $\t=\norm{\zeta}^{-(1-\varrho)}$, the claim follows upon noting that $\norm{\zeta}^{-\varrho}\geq \t^{\varrho}=\norm{\zeta}^{-\varrho(1-\varrho)}$.
    \end{proof}

We are now in position to complete the proof of Proposition \ref{prop:Case 1 Veronese all dims multiple contractions}.

    Let $C\geq 1$ be such that the supports of all the measures $\g_\l \ast\check{\g}_\l$ are contained a ball of radius $C$ around the origin. Let $\l\in \L_\z$. 
    By Corollary ~\ref{Coro Tsujii}, for every $\e>0$ there is $\d>0$ so that for $|z|\leq C  \norm{\l\zeta} $, 
    $$|\hat{\mu}(z)|\ll \norm{\l \zeta}^{-\d}$$
    except for a set $B \subseteq \mathbb{R}$ of frequencies $z$ that is a union of $O_\e(\norm{\l \zeta}^\e)$ intervals of length $1$. Since every $\l,\l' \in \L_\z$ are $\l \asymp \l'$, we may assume this property (and $B$ in particular) is independent of the choice of $\l$.

    Fix  $\e =\b/2$, for $\b$ as in Lemma~\ref{lem:apply discrete friendly}.
    Hence, recalling that $\l \asymp \norm{\zeta}^{-(1-\varrho)}$ for all $\l\in\L_\zeta$, via Lemma \ref{Lemma first bound} and the bounds above, we obtain
    \begin{align*}
        |\hat{\nu}(\xi)|^{2}
        \leq      
        \#\L_\zeta^2 \times \sum_{\l\in \L_\zeta}
        \g_\l\ast \check{\g}_\l\left(\set{y: \l \zeta\cdot y \in B}\right)
        +  O(\# \L_\zeta^3 \times\norm{\zeta}^{-\varrho\d}+ \# \L_\zeta^2  \times \norm{\zeta}^{2\varrho-1}).
    \end{align*}

    Fix a unit-length interval $I=(a-1/2,a+1/2)\subseteq B$ for some $a\in \R$, and consider the affine hyperplane $W_a = \set{x: \l\zeta \cdot x=a}$. An elementary computation then shows that  if $y$ is such that $\l\zeta \cdot y\in I$, then $y\in W_a(\norm{\l\zeta}^{-1})$.
    Thus, noting that 
    $$\g_\l \ast \check{\g}_\l(A) = \int \g_\l(A+z) \,d\g_\l(z)  \text{ for any Borel set } A,$$ 
    we obtain
    \begin{align*}
        \g_\l \ast \check{\g}_\l( \set{y: \l\zeta\cdot y \in I}) \leq  \sup_{ W=W_a+z}\g_\l(W(\norm{\l\zeta}^{-1})) \leq C\norm{\l\zeta}^{-\b},
    \end{align*}
    where the supremum runs over all translations of $W_a$, and we applied Lemma \ref{lem:apply discrete friendly} in the last inequality.

    Putting together the above estimates, we arrive at the bound
    \begin{align*}
        |\hat{\nu}(\xi)|^{2}
        \ll      
        \#\L_\zeta^3 \times \norm{\zeta}^{(\e-\b)\varrho}+ \# \L_\zeta^3 \times\norm{\zeta}^{-\varrho\d}+ \# \L_\zeta^2 \times \norm{\zeta}^{2\varrho-1}.
    \end{align*}
    Finally, by Lemma~\ref{lem:cut set}, we have $\#\L_\zeta \ll (\log\norm{\zeta})^A$, where $A=\# \Phi+1$.
    Recalling that $\e=\b/2$ and $\varrho<1/2$, this bound completes the proof.
    \qedhere
\end{proof}

    \subsection{General non-affine curves, and proof of Proposition~\ref{prop:Case 1 general curves}}
    \label{sec:case 1 general curves}
    The goal of this section is to deduce Proposition~\ref{prop:Case 1 general curves} from
    Proposition~\ref{prop:Case 1 Veronese all dims multiple contractions}. 
    The deduction is via Taylor expansion and a change of coordinates. 
    Recall the notation $V_{d}$ set before Proposition \ref{prop:Case 1 Veronese all dims multiple contractions}.

\begin{proof}[\textbf{Proof of Proposition~\ref{prop:Case 1 general curves}}]
    
We give the proof in the case $Q(x)$ is a non-trapped analytic curve. 
The proof in the case $Q$ is a $C^{d+1}$ non-degenerate curve is very similar, and in fact simpler, due to the non-vanishing of the determinant in \eqref{eq:nondegenerate} over an entire neighborhood of $\supp(\mu)$.

As usual, by replacing $\mu$ with an affine image of itself, we shall assume its support is contained in $[0,1]$.
Let $\Phi=\set{f_i:f_i(x)=\l_ix+t_i }_{i\in I}$.
Let $\a>0$ be the exponent provided by Proposition~\ref{prop:Case 1 Veronese all dims multiple contractions}.
Fix a frequency $\xi = (\th,\zeta)\in \R\times \R^{d-1}$, with $\norm{\zeta}>1$, and define
\begin{align*}
    \t = \norm{\zeta}^{-(1+\a)/(1+d(1+\a))}.
\end{align*}
Let $P_\t$ be the cut-set defined in Def.~\ref{def:cut sets}.
Without loss of generality, we shall assume over the course of the proof that $\norm{\zeta}$ is sufficiently large, depending only on $g$ and $\mu$.

Let $\varrho >0$ be the parameter provided by Proposition~\ref{prop:mu is Frostman}.
Let $\e =\e(\a, g,d) \in (0,\varrho)$ to be chosen at the end of the proof to be sufficiently small depending only on $g$, the ambient dimension $d$, and the exponent $\a$, and let 
\begin{align*}
    \d= \t^\e.
\end{align*}   
Let  $E = E(\d) \subset [0,1]$ be the set provided by Lemma~\ref{lem:good intervals higher dimensions} and let 
$$P'_\t=\set {\w \in P_\t:\, f_\w(0)\in E }.$$  
For a word $\w\in I^\ast$, we write $\nu_\w$ for the pushforward of $f_\w \mu$ under $x\mapsto (x,g(x))$.
\begin{lem} \label{lemma away from degeneracy}
For the Frostman exponent $s_0$ of $\mu$ we have
\begin{align*}
       |\widehat{\nu}(\xi)|
    \leq   \sum_{\w\in P_\t\setminus P_\t'}  \mbf{p}_\w |\widehat{\nu_\w}(\xi)| +
     O_{g,\mu} \left( \d^{s_0} \right).
\end{align*}
\end{lem}
\begin{proof}
 By stationarity of $\mu$ (Lemma~\ref{lem:cut set}) and the triangle inequality, we have
\begin{align*}
    |\widehat{\nu}(\xi)|
    &\leq  \sum_{\w\in P_\t\setminus P_\t'}  \mbf{p}_\w |\widehat{\nu_\w}(\xi)| + \sum_{\w\in  P_\t'} \mbf{p}_\w
    .
\end{align*}
By Lemma \ref{lem:good intervals higher dimensions}, $E$ is a union of $O_g(1)$ $\d$-intervals. Therefore, by Proposition~\ref{prop:mu is Frostman}, since $\d>\t^\varrho$, for the Frostman exponent $s_0$ of $\mu$ we have
\begin{align*}
       |\widehat{\nu}(\xi)|
    \leq  \sum_{\w\in P_\t\setminus P_\t'}  \mbf{p}_\w |\widehat{\nu_\w}(\xi)| + 
     \sum_{\w\in P_\t'} \mbf{p}_\w 
     = \sum_{\w\in P_\t\setminus P_\t'}  \mbf{p}_\w |\widehat{\nu_\w}(\xi)| +
     O_{g,\mu} \left( \d^{s_0} \right).
\end{align*}
The proof is complete.
\end{proof}

Thus, it remains to estimate the sum over $P_\t\setminus P_\t'$.
\begin{lem} \label{lemma intermediate step}
If $\t$ is sufficiently small then for every $\w\in P_\t$ we have, writing $f_\w(x)=\l_\w x+t_\w$, 
\begin{align*}
     \left|\widehat{\nu}_\w(\xi)\right| \leq 
\left|\int e( \langle \xi_\w, V_d(\l_\w x) \rangle)\,d \mu(x)\right| + O(\norm{\zeta} \t^{d+1}).
\end{align*}
Here, given a frequency $\xi=(\th,\zeta)\in \R\times \R^{d-1}$, we define $\xi_\w := (\th_\w, \zeta_\w)$, where
\begin{align*}
    \th_\w := \th+\langle\zeta, g'(t_\w)\rangle,
    \qquad 
    \zeta_\w: = \left( \langle \zeta , g''(t_\w)\rangle /2, \dots, \langle\zeta, g^{(d)}(t_\w)\rangle /d!
    \right).
\end{align*}
\end{lem}
\begin{proof}
Note that  
$$\mrm{diam}(\supp(f_\w \mu)) \asymp \t \text{ for every } \omega \in P_\t. $$
We  assume  $\t$ is small enough so that $g$ can be written as a power series in a neighborhood of $t_\w$ that contains $\supp(f_\w\mu)$. By considering the Taylor expansion of $g$ about $t_\w$, we see that for every $x\in \supp(f_\w\mu)$,
\begin{align*}
    \langle \xi, (x,g(x))\rangle 
    = \zeta \cdot g(t_\w) + \xi_\w \cdot V_d(x-t_\w) + O(\norm{\zeta} \t^{d+1}),
\end{align*}
where we recall that $V_d(y) = (y,y^2,\dots,y^d)$ is the moment curve.
Hence, by the Lipchitz continuity of the Fourier transform,  we obtain the bound
\begin{align*}
    \left|\widehat{\nu}_\w(\xi)\right| \leq 
\left|\int e( \langle \xi_\w, V_d(x-t_\w) \rangle)\,df_\w \mu(x)\right| +O(\norm{\zeta} \t^{d+1}).    
\end{align*}
Recalling that $f_\w(x)=\l_\w x+t_\w$ thus yields
\begin{align*}
     \left|\widehat{\nu}_\w(\xi)\right| \leq 
\left|\int e( \langle \xi_\w, V_d(\l_\w x) \rangle)\,d \mu(x)\right| + O(\norm{\zeta} \t^{d+1}).    
\end{align*}
This proves the Lemma
\end{proof}

Let $x\in [0,1]$ and $\omega \in P_\t$. We define two $(d-1)\times (d-1)$ matrices $G(x),D_\omega$ via 
$$G(x) \text{ is the } ((d-1)\times (d-1)) \text{ matrix with } k^{th}  \text{ column given by the vector } g^{(k+1)}(x)$$ 
and 
$$D_\w \text{ is the } (d-1)\times (d-1) \text{ diagonal matrix with } k^{th} \text{ diagonal entry } \l^{k+1}_\w/(k+1)!.$$
Let $F_\omega (x)= D_\w\cdot G(x)$. When $\omega$ is fixed we simply write $F(x)$. Note that 
$$\langle \xi_\w, V_d(r_\w x) \rangle = \langle (\l_\w\th_\w, \zeta\cdot F(t_\w) ), V_d( x) \rangle.$$
Hence, by Proposition~\ref{prop:Case 1 Veronese all dims multiple contractions} and Lemma \ref{lemma intermediate step}, we obtain for some $\a>0$
\begin{align}\label{eq:reduce to Veronese}
    \left|\widehat{\nu}_\w(\xi)\right| \ll 
\norm{\zeta \cdot F(t_\w)}^{-\a} + \norm{\zeta} \t^{d+1}.
\end{align}
Recall the definition of $P_{\t'}$ from before Lemma \ref{lemma away from degeneracy}.
\begin{lem}\label{claim:lower bound by determinant}
    There is $C\geq 1$, depending only on $g$, such that for every $\w\in P_\t\setminus P_\t'$
    \begin{align*}
    \norm{\zeta \cdot F(x)}\geq \norm{\zeta \cdot D_\w} |\det(G(x)|/C ,    
    \end{align*}
    uniformly over $x\in [0,1]$ and $\zeta\in \R^{d-1}$, where  $\det(-)$ is the matrix determinant.
\end{lem}
\begin{proof}
    Let $\s_{\min}(x)$ denote the  singular value of smallest magnitude of $G(x)$. Then 
    $$\norm{\zeta\cdot F(x)}\geq \s_{\min}(x) \norm{\zeta \cdot D_\w}.$$
    Note that $\s_{\min}(x)$ is non-zero in view of Lemma~\ref{lem:good intervals higher dimensions} and the choice of $\omega \in P_{\t'}$.
    Moreover, 
    $$\s_{\min}(x) \geq |\det(G(x)|/C(x), \text{ where } C(x)= \text{ the product of all the other singular values of } G(x).$$
    Hence, since the singular values of $G(x)$ are non-vanishing and vary continuously with $x$, $C(x)$ is uniformly bounded above over $x\in [0,1]$. This implies the Lemma.
\end{proof}

We can now complete the proof. Let $\omega \in P_{\t} \setminus P_{\t'}$. Note that $\norm{\zeta \cdot D_\w} \gg \norm{\zeta} \l_\w^{d}$.
Hence, by~\eqref{eq:reduce to Veronese}, Lemma ~\ref{claim:lower bound by determinant} and Lemma~\ref{lem:good intervals higher dimensions}, we obtain
\begin{align*}
      \left|\widehat{\nu}_\w(\xi)\right| 
      \ll \d^{-\a c_1} \norm{\zeta \cdot D_\w}^{-\a} + \norm{\zeta} \t^{d+1}
      \ll \d^{-\a c_1} \l_\w^{-d \a} \norm{\zeta}^{-\a} + \norm{\zeta} \t^{d+1}
    ,
\end{align*}
where $c_1 =c_1(g)\geq 1$ is the exponent provided by Lemma~\ref{lem:good intervals higher dimensions}.
Recalling that $\l_\w \asymp \t $, we obtain
\begin{align*}
      \left|\widehat{\nu}_\w(\xi)\right| 
      \ll
      \d^{-\a c_1} \t^{-d \a } \norm{\zeta}^{-\a} 
      + \norm{\zeta} \t^{d+1}
      \leq \d^{-\a c_1} (\t^{-d \a } \norm{\zeta}^{-\a} 
      + \norm{\zeta} \t^{d+1})
      .
\end{align*}
Setting $\b= \a/(d(1+\a)+1)$, by our choice of $\t$ the above bound becomes
\begin{align*}
    \left|\widehat{\nu}_\w(\xi)\right| 
      \ll
      \d^{-\a c_1} \norm{\zeta}^{-\b}.
\end{align*}
Recall that $\d=\t^\e$.
Thus, taking $\e$ sufficiently small, depending on $\a$ and $c_1$, we can ensure that $\d^{-\a c_1} \norm{\zeta}^{-\b}$ is at most $\norm{\zeta}^{-\b/2}$. 
Hence, the result follows by combining the above bound with Lemma \ref{lemma away from degeneracy}.
\qedhere
\end{proof}

\subsection{Average decay away from the first coordinate, and proof of Theorem~\ref{thm: case 1}}

Proposition~\ref{prop:Case 1 general curves} yields the following  Corollary, which is a more precise form of Theorem~\ref{thm: case 1}.  

\begin{cor}\label{cor:integrate Case 1 bound}
    Let $\mu, g,$ and $\nu$ be as in Theorem~\ref{thm: case 1}.
    For $R\geq 1$ and $\e>0$,
    let 
    \begin{align*}
        C_{R,\e} =\set{(\th,\zeta)\in \R\times \R^{d-1}: |\th|^\e \leq \norm{\zeta}\leq R}.
    \end{align*}
    Let $\g>0$ be the parameter provided by Proposition~\ref{prop:Case 1 general curves}.
    Then, for $p>d^3/\g\e^2$,
    \begin{align*}
        \int_{C_{R,\e}} |\widehat{\nu}(\xi)|^p\;d\xi =O(R^\e).
    \end{align*}
\end{cor}

\begin{proof}
Let $\d= \e/d$.
Let $D_{R,\d} = C_{R,\d}\setminus \set{\xi: \norm{\xi}\leq R^\d}$.
Then, applying Proposition~\ref{prop:Case 1 general curves} on $D_{R,\e}$,  using its notations and the trivial bound $|\widehat{\nu}(-)|\leq 1$ on its complement, we get
\begin{align*}
    \int_{C_{R,\e}} |\widehat{\nu}(\xi)|^p\,d\xi 
    = \int_{D_{R,\e}} |\widehat{\nu}(\xi)|^p\;d\xi 
    + O(R^\e)
     \ll \int_{\set{\xi=(\th,\zeta)\in D_{R,\d}}} \norm{\zeta}^{-\g p} \,d\xi + R^\e.
\end{align*}
Finally, note that, on $D_{R,\d}$, $\norm{\zeta}\gg \norm{\xi}^{\d} > R^{\d^2}$.
The corollary follows since $D_{R,\d}$ has measure $O(R^d)$.
\qedhere

\end{proof}

\section{Average Decay Near the First Coordinate: Proof of Proposition~\ref{prop: Case 2}} \label{Section:decay near first coordinate}

The purpose of this Section is to prove the following estimate on the average decay of curved self-similar measures for frequencies with dominant first coordinate.
This immediately yields Proposition~\ref{prop: Case 2}, which is the second main ingredient in the proof of Theorem~\ref{main theorem general}.
We recall the statement of that proposition for the reader's convenience.
\begin{prop} \label{prop:Case 2 general curves all dimensions}
Let $\mu \in \mathcal{P}(\mathbb{R})$ be a non-atomic  self-similar measure and $d\geq 2$. 
Let $g:U\r \R^{d-1} $ be a  $C^1$-map defined on an open neighborhood $U$ of $\supp(\mu)$.
   For $Q(x)=(x,g(x))$, 
    let $\nu =Q\mu$ be the pushforward of $\mu$ to the graph of $g$.
    For $R\geq 1$ and $\e>0$, let 
    \begin{align*}
        E_{R,\e} =\set{(\th,\zeta)\in \R\times \R^{d-1}: \norm{\zeta}\leq |\th|^\e \leq  R^\e}.
    \end{align*}
Then, for every $\e >0$, there exists  $p=p(\e, \mu)>1$ such that for all $R\geq 1$, we have
\begin{align*}
    \int_{\xi\in E_{R,\e}}\left| \widehat{\nu} (\xi) \right|^{p} \, d\xi = O_{\e,\mu}(R^\e).
\end{align*}
\end{prop}

The rest of this section is dedicated to the proof of Proposition~\ref{prop:Case 2 general curves all dimensions}.
By replacing $\mu$ with an affine image of itself, we shall assume its support is contained in $[0,1]$.
    Let $\Phi=\set{f_i(x)=\l_ix+t_i }_{i\in I}$.
As before, for a word $\w\in I^\ast$, we write $\nu_\w$ for the pushforward of $f_\w \mu$ under $x\mapsto (x,g(x))$.
Recall the definition of the cut-set $P_\tau$ from Definition \ref{def:cut sets}, and the notations set in Section \ref{Section self-similar prel.}.

Fix $\e>0$ and let $\d=\e/d$.
We also fix $R>1$, which we shall assume to be sufficiently large depending on $\e$ and $\mu$, and let
\begin{align*}
    F_{R,\e} = E_{R,\e} \setminus \set{(\th,\zeta)\in E_{R,\e}: |\th|\leq R^\d}.
\end{align*}
First, note that since $E_{R,\e}\setminus F_{R,\e}$ has measure $O(R^{\e})$ and $|\hat{\nu}(-)|\leq 1$, we get
    \begin{align}\label{eq:Case 2 excise origin}
        \int_{E_{R,\e}} \left|\widehat{\nu}(\xi)\right|^p \,d\xi 
        \leq \int_{F_{R,\e}} \left|\widehat{\nu}(\xi)\right|^p \,d\xi 
        + O(R^\e).
    \end{align}
Hence, we focus on the region $F_{R,\e}$.
By Lemma \ref{lem:cut set}, for every $\tau>0$ and every $\xi \in \mathbb{R}^d$ we have
$$\hat{\nu}( \xi)=\sum_{\omega \in P_{\tau}}\mbf{p}_\w\widehat{\nu_\omega}( \xi ).$$
We have the following initial pointwise bound on terms of the above sum using the Fourier transform of the original self-similar measure.
\begin{lem}\label{lem:bound by first coordinate}
    For every $\t\in (0,1)$, $\w\in P_\t$, and $\xi=(\th,\zeta)\in \R\times \R^{d-1}$, we have
    \begin{align*}
        \left| \widehat{\nu_\omega}(\xi) \right| 
    \leq \left|\hat{\mu}(\l_\omega \th)\right| 
    +O(  \t \norm{\zeta})     .
    \end{align*}
\end{lem}
\begin{proof}
Note that, for every $\omega\in P_\tau$ and $x\in \supp(\mu)$,  $\l_\omega \leq \t$, since $g$ is $C^1$, we have $g(f_\w(x))=g(f_\w(0)) + O(\t)$.
Thus, since $f_\omega(x)=\l_\omega \cdot x + t_\omega$, and $\hat{\mu}(-)$ is Lipschitz continuous, we get for every frequency $\xi =(\th,\zeta)\in \R\times \R^{d-1}$ that
\begin{align*}
    \left| \widehat{\nu_\omega}(\xi) \right| 
    = \left |\int e(  (\th f_\omega (x)+ \zeta\cdot g(t_\w) )\,d\mu(x) \right| + O(\norm{\zeta}\t)
    = \left|\hat{\mu}(\l_\omega \th)\right| 
    +O(  \t \norm{\zeta})     ,
\end{align*}
which is the claimed bound.
\end{proof}

The rest of the proof of Proposition~\ref {prop:Case 2 general curves all dimensions} is broken down into the following two Lemmas. 
In what follows, for $n\in \N$, we let for
\begin{align}\label{eq:Case 2 parameters}
    q \stackrel{\mrm{def}}{=} 2^{1/2\e}, \qquad
    \t_n \stackrel{\mrm{def}}{=} 2^{-n},
    \qquad
    F_{R,\e}^n \stackrel{\mrm{def}}{=}
    \set{(\th,\zeta)\in F_{R,\e}: q^n\leq |\th|< q^{n+1}}.
\end{align}

\begin{lem} \label{lem:annular partition of Case 2}
For all integers $p\geq 2d^2/\e$, we have
\begin{align*}
    \int_{F_{R,\e}} \left|\widehat{\nu}(\xi)\right|^p \,d\xi 
    \leq
\sum_{n\in \N} 
\int_{F_{R,\e}^n}
\left( \sum_{\w \in P_{\t_n}} \mbf{p}_\w \left|\hat{\mu}(\l_\w \th(\xi)) \right|^p \right) \, d\xi + O_\e(1),
\end{align*}
where $\th(\xi)$ denotes the first coordinate of $\xi$.
\end{lem}
\begin{proof}
First, since $F_{R,\e} = \cup_{n\in\N} F_{R,\e}^n$, we get
\begin{align*}
    \int_{F_{R,\e}} \left|\widehat{\nu}(\xi)\right|^p \,d\xi 
    &\leq \sum_{n\in\N} \int_{F_{R,\e}^n} \left|\widehat{\nu}(\xi)\right|^p \,d\xi 
    = \sum_{n\in\N} \int_{F_{R,\e}^n} 
    \left| \sum_{\w \in P_{\t_n} }\mbf{p}_\w\widehat{\nu_\omega}( \xi) \right|^{p} \,d\xi.
\end{align*}
Note that the outer sum runs over $n$ with $R^\d \leq q^{n+1} \leq q R $.
Applying Lemma~\ref{lem:bound by first coordinate} with $\t=\t_n$ for each $n$, we obtain
\begin{align*}
    \int_{F_{R,\e}} \left|\widehat{\nu}(\xi)\right|^p \,d\xi 
    &\leq 
    \sum_{n\in \N} \int_{\xi=(\th,\zeta) \in F_{R,\e}^n}  \left| \sum_{\omega \in  P_{\t_n} } \mbf{p}_\w \left[ \left| \hat{\mu}(\l_\omega\cdot  \th) \right|+ O(2^{-n} \norm{ \zeta}) \right]\right|^{p} \, d\xi.
\end{align*}
By Jensen's inequality, since $p>1$ and $\sum_{\omega \in P_{\t_n}} \mbf{p}_\w=1$, we obtain
\begin{align*} 
\left| \sum_{\omega \in  P_{\t_n} } \mbf{p}_\w \left[ \left| \hat{\mu}(\l_\omega\cdot  \th) \right|+ O(2^{-n} \norm{ \zeta}) \right]\right|^{p}
&\ll  \norm{\zeta}^p \cdot 2^{-np}   + \sum_{\omega \in P_{\t_n} }\mbf{p}_\w \left| \hat{\mu}(\l_\omega \cdot \th) \right|^{p}
.
\end{align*}

Integrating the second term over $F_{R,\e}^n$ and summing over $n\in \N$ gives the first term of the claimed bound.
For the first term, note that for each $\xi=(\th,\zeta)\in F_{R,\e}^n$, we have that $\norm{\zeta}\leq |\th|^{\e } \leq q^{\e(n+1)}$.
The choice of $q$ in~\eqref{eq:Case 2 parameters} also gives $q^{\e (n+1)} 2^{-n}\leq q^\e 2^{-n/2}$.

Thus, using that $F_{R,\e}$ has measure $O(R^d)$ and that the sum over $n$ only involves terms satisfying $R^\d\leq q^{n+1}\leq R q$, we get
\begin{align*}
    \sum_{n\in \N} \int_{\xi=(\th,\zeta) \in F_{R,\e}^n}
    \norm{\zeta}^p \cdot 2^{-np}  \,d\xi
    \ll \sum_{n\in \N:  q^{n+1}\geq R^{\d}}
    2^{-np/2} R^d
    \leq R^{d -\d p/2} q
    .
\end{align*}
Taking $p\geq 2d/\d=2d^2/\e$, the above bound is $O_\e(1)$, thus completing the proof.
\end{proof}

The next lemma estimates the main term in the bound provided by Lemma~\ref{lem:annular partition of Case 2}.
The key ingredient is Corollary~\ref{Coro Tsujii} on $L^2$-flattening of self-similar measures.
\begin{lem} \label{Lemma 3 rewrite}
For all $p\gg_{\e,\mu} 1$, we have
\begin{align*}
\sum_{n\in \N} 
\int_{F_{R,\e}^n}
\left( \sum_{\w \in P_{\t_n}} \mbf{p}_\w \left|\hat{\mu}(\l_\w \th(\xi)) \right|^p \right) \, d\xi 
= O_{\e,\mu}(R^{\e(d+3)}),
\end{align*}
where $\th(\xi)$ denotes the first coordinate of $\xi$.
\end{lem}
\begin{proof}
Recalling the definition of $F_{R,\e}^n$ in~\eqref{eq:Case 2 parameters}, since the integrand depends only on the first coordinate of $\xi$, we get
\begin{align}\label{eq:integrate zeta away in Case 2}
    \sum_{n\in \N} 
    \int_{F_{R,\e}^n}
    \left( \sum_{\w \in P_{\t_n}} \mbf{p}_\w \left|\hat{\mu}(\l_\w \th(\xi)) \right|^p \right) \, d\xi 
    \ll \sum_{n\in \N: R^\d \leq q^{n+1} \leq Rq} q^{\e (n+1)(d-1)} 
    \sum_{\w \in P_{\t_n}} \mbf{p}_\w 
    \int_{q^n}^{q^{n+1}} 
     \left|\hat{\mu}(\l_\w \th) \right|^p  \, d\th ,
\end{align}
where, as in the proof of the previous lemma, we also used the fact that the outer sum  on the left-hand side runs over $n\in\N$ satisfying $R^\d \leq q^{n+1} \leq Rq$.

Using the change of variable $\l_\w \th\mapsto \th$, and recalling that $\l_\w\asymp \t_n=2^{-n}$ for all $\w\in P_{\t_n}$, we get
\begin{align*}
    \eqref{eq:integrate zeta away in Case 2}
    &\ll \sum_{n\in \N: R^\d \leq q^{n+1} \leq Rq} q^{\e (n+1)(d-1)} 
    \sum_{\w \in P_{\t_n}} \mbf{p}_\w 
    \int_{|\th| \asymp (q/2)^n}
     \left|\hat{\mu}( \th) \right|^p 2^n \, d\th
    \nonumber\\
    &\ll \sum_{n\in \N: R^\d \leq q^{n+1} \leq Rq} q^{\e (n+1)(d-1)} \times  2^n \times
    \int_{|\th| \asymp (q/2)^n}
     \left|\hat{\mu}( \th) \right|^p  \, d\th .
\end{align*}
Now, by Corollary~\ref{Coro Tsujii}, if $p$ is large enough, depending only on $\mu$ and $\e$, for each $n$, the inner integral is $O_{\mu,\e}((q/2)^{\e n})$ .
Thus, recalling that $2=q^{2\e}$, we arrive at the bound
\begin{align*}
    \eqref{eq:integrate zeta away in Case 2}
    &\ll_{\e,\mu} 
    \sum_{n\in \N:  q^{n} \leq R}
    q^{\e nd} \times  2^n \times q^{\e n}
    \ll R^{\e(d+3)} ,
\end{align*}
which completes the proof of the lemma.
\end{proof}

Since $\e$ was arbitrary,
Proposition~\ref {prop:Case 2 general curves all dimensions} now follows from the combination of~\eqref{eq:Case 2 excise origin} with Lemmas \ref{lem:annular partition of Case 2} and \ref{Lemma 3 rewrite}.

\bibliographystyle{amsalpha}
\bibliography{bib}

\end{document}